\documentclass[eqno,12pt]{article}
\usepackage{amsfonts,amssymb}
\usepackage{amssymb,amsmath,latexsym}
\usepackage{amsthm}
\usepackage[varg]{pxfonts}
\usepackage{bbm}
\usepackage{CJK}
\usepackage{fancyhdr}
\usepackage{graphicx}
\usepackage{geometry}
\usepackage{psfrag}
\usepackage{time}
\usepackage{colortbl}
\usepackage{stmaryrd}
\usepackage{mathrsfs}

\usepackage{indentfirst,latexsym,bm}

\makeatletter \@addtoreset{equation}{section}

\newtheorem{prop}{Proposition}[section]
\newtheorem{thm}[prop]{Theorem}
\newtheorem{defn}{Definition}[section]

\newtheorem{corol}[prop]{Corollary}

\newtheorem{examp}{Example}[section]
\newtheorem{rem}{Remark}[section]

\begin{document}

\baselineskip 5true mm
\begin{center}
 {\Large\bf  On semi-Quasi-Einstein Manifold
 \footnote{This work is supported by National Natural Science Foundation of China (No.11871275; No.11371194) and Ministry of Education, Malaysia grant FRGS/1/2019/STG06/UM/02/6.}}
\end{center}
\centerline{ Yanling Han }
\begin{center}
{\scriptsize School of Mathematics and Statistics, Qilu University
of Technology(Shandong Academy of Sciences), \\Jinan 250353,
P.R.China,\\
 E-mail: hanyanling1979@163.com}
\end{center}
\centerline{Avik De}
\begin{center}
{\scriptsize  Department of Mathematical and Actuarial Sciences, Universiti Tunku Abdul Rahman, Bandar sungai Long, \\Kajang 43000, Malaysia\\
 E-mail: de.math@gmail.com}
\end{center}
\centerline{ Peibiao Zhao}
\begin{center}
{\scriptsize  Dept. of Applied Mathematics, Nanjing University of Science and Technology, Nanjing 210094, P. R. China\\
E-mail: pbzhao@njust.edu.cn}
\end{center}

\begin{center} 
\begin{minipage}{128mm}
{\footnotesize {\bf Abstract}  In the present paper we introduce a semi-quasi-Einstein manifold from a semi symmetric metric connection. Among others, the popular Schwarzschild and Kottler spacetimes are shown to possess this structure. Certain curvature conditions are studied in such a manifold with a Killing generator.  \\
{\bf Keywords} Quasi-Einstein Manifold; Semi-symmetric metric
connection; Ricci-symmetric, Kottler spacetime, Schwarzschild spacetime. \\
{\bf MR(2000)} 53C20, 53D11.}
\end{minipage}
\end{center}

\section{ Introduction}

Ricci curvature plays a significant role both in geometry and in the
theory of gravity. If the Ricci curvature of some space satisfies
$S(X,Y)=ag(X,Y)$ for scalar $a$, we call
 it Einstein manifold as such a space satisfies the Eintein's field equations in vacuum.
 Einstein space is naturally generalized to a wider range, such as quasi-Einstein space \cite{CMC1},
 generalized quasi-Einstein space \cite{CMC2}, mixed quasi-Einstein space and so on, by adding some
 curvature constraints to the Ricci tensor.

A quasi-Einstein manifold is the closest and simplest approximation
of the Einstein manifold. It evolved as an exact solution of
einstein's field equations and also in geometry while studying
 quasi-umbilical hypersurfaces. A non-flat Riemannian manifold $(M^{n},g) (n>2) $ is defined to be a
quasi-Einstein manifold if the Ricci tensor is not identically zero
and satisfies the condition
\begin{equation*}
S(X,Y)=ag(X,Y)+b\eta(X)\eta(Y),\forall X,Y\in TM
\end{equation*}
for some scalar functions $a,b\neq 0$, where $\eta$ is a non-zero 1-form such
that $g(X,\xi)=\eta(X),$ $\eta$ is called the associated
1-form and $\xi$ is called the generator of the manifold.

It is also known that the perfect fluid space-time in general
 relativity is a 4-dimensional semi-Riemannian quasi-Einstein
space. The study of these space-time models can help us better
understand the evolution of the universe. Quasi-Einstein manifold was defined by M. C. Chaki and R. K.
Maity\cite{CMC1}, later it was widely studied in
\cite{Be,CM,CS,DD1,DD2,DG1,DG2,FHZ}. Z.L.Li \cite{Li1} discussed the
basic properties of quasi-Einstein manifold, obtained some geometric
characteristics and proved the non-existence of some quasi-Eistein
manifolds. U.C.De and G. C. Ghosh \cite{DG1} proved
a theorem for existence of a QE-manifold and give some examples
about QE-manifolds. V.A.Kiosak \cite{Ki} considered the conformal
mappings of quasi-Einstein spaces and proved that they are closed
with respect to concircular mappings. M.M. Tripathi and J.S. Kim
\cite{TK} studied a quasi-Einstein manifold whose generator belongs
to the $k$-nullity distribution $N(k)$ and have proved that
conformally quasi-Einstein manifolds are certain $N(k)$-quasi
Einstein manifolds.

The purpose of this paper is to investigate quasi-Einstein manifolds
from a rather different perspective. By replacing the Levi-Civita
connection with a semi-symmetric metric connection \cite{Ya}, we
define a semi-quasi-Einstein manifold which generalizes the concept
of quasi-Einstein manifolds. Semi-symmetric
 metric connections are widely used to modify Einstein's gravity theory in form of Einstein-Cartan gravity
 and also in unifying the gravitational and electromagnetic forces. Recently in \cite{semi-rel} the field equations
 are derived from an action principle which is formed by the scalar curvature of a semi-symmetric metric connection.
  The derived equations contain the Einstein and Maxwell equations in vacuum. In fact, Schouten criticized
  Einstein's argument for using a symmetric connection and used the notion of semi-symmetric connections in his approach \cite{schouten}.

We further prove the existence of such structure by analyzing the Schwarzschild and Kottler spacetimes and obtain a necessary and sufficient conditions for such a manifold to be an Einstein space. 

In \cite{Mura} Murathan and Ozgur considered the semi-symmetric
metric connection with unit parallel vector field $P$, and  proved
that $R.\bar{R}=0$ if and only if $M$ is semi-symmetric; if
$R.\bar{R}=0$ or $R.\bar{R}-\bar{R}.R=0$ or $M$ is semi-symmetric and
$\bar{R}.\bar{R}=0$, then $M$ is conformally flat and
quasi-Einstein. Motivated by this we further investigate the manifolds for Ricci symmetric and Ricci semi-symmetric criterion under similar ambience.

The paper is organized as follows: Section 2 recalls the form and
some curvature properties of semi-symmetric metric connection.
Section 3 gives the main definition of semi-quasi-Einstein
manifold ($S(QE^{n})$) and discusses the relations between
$S(QE^{n})$, quasi-Einstein and generalized quasi-Einstein
manifolds. Section 4 concerns with some physical and geometric
characteristics of $S(QE^{n})$ manifold under certain curvature
conditions. Some interesting examples of $S(QE^{n})$ manifold is constructed in the
last section.

 \section{ Preliminaries} \setcounter{section}{2}
 \setcounter{equation}{0}

Let $(M^{n},\nabla)$ be a Riemannian manifold of dimension $n$ and
$\nabla$ be the Levi-Civita connection compatible to the metric $g$.
A semi-symmetric metric connection is defined by
\begin{equation*}
\bar{\nabla}_{X}Y=\nabla_{X}Y+\pi(Y)X-g(X,Y)P, \quad  \forall X,Y\in
TM
\end{equation*}
where $\pi$ is any given 1-form and $P$ is the associated vector
field, $g(X,P)=\pi(X)$.

By direct computations, one can obtain
\begin{eqnarray}\label{21}
&\bar{R}(X,Y)Z=R(X,Y)Z+g(AX,Z)Y-g(AY,Z)X\nonumber\\
&+g(X,Z)AY-g(Y,Z)AX,\notag
\end{eqnarray}
\begin{eqnarray}\label{22}
&g(AY,Z)=(\nabla_{Y}\pi)(Z)-\pi(Y)\pi(Z)+\frac{1}{2}\pi(P)g(Y,Z)\nonumber\\
&=g(\nabla_{Y}P,Z)-\pi(Y)\pi(Z)+\frac{1}{2}\pi(P)g(Y,Z),\notag
\end{eqnarray}
\begin{eqnarray}\label{23}
&g(\bar{\nabla}_YP,Z)-g(\bar{\nabla}_ZP,Y)=g(\nabla_YP,Z)-g(\nabla_ZP,Y)\notag
\end{eqnarray}
\begin{eqnarray}\label{24}
&\bar{S}(Y,Z)=S(Y,Z)-(n-2)(\nabla_{Y}\pi)(Z)+(n-2)\pi(Y)\pi(Z)\nonumber\\
&-\{(n-2)\pi(P)+\sum_{i}g(\nabla_{i}P,e_{i})\}g(Y,Z),
\end{eqnarray}
further one has
\begin{eqnarray*}
\bar{R}(X,Y)Z&=&-\bar{R}(Y,X)Z,\notag\\
\bar{R}(X,Y)Z+\bar{R}(Y,Z)X+\bar{R}(Z,X)Y&=&(g(AZ,Y)-g(AY,Z))X\notag\\
&&+ g(AX,Z)-g(AZ,X))Y\notag\\
&&+(g(AY,X)-g(AX,Y))Z.\notag
\end{eqnarray*}
It is obvious that the Ricci tensor of the semi-symmetric metric
connection is not symmetric unless under certain conditions. For
example, if $\pi$ is closed, then $A$ is symmetric, thus
$\bar{S}$ is also symmetric.

In particular, if a unit $P$ satisfies Killing's equation $g(\nabla_XP,Y)+g(\nabla_YP,X)=0$, then
\begin{eqnarray}\label{2}
&&g(\nabla_{P}P,Y)=0,\quad g(\nabla_XP,X)=0,\quad AX=\nabla_{X}P-\pi(X)P+\frac{1}{2}X,\notag\\
&&\bar{R}(X,Y)P=R(X,Y)P+\pi(Y)\nabla_{X}P-\pi(X)\nabla_{Y}P\notag\\
&&\bar{R}(X,P)Y=R(X,P)Y+g(\nabla_{X}P,Y)-\pi(Y)\nabla_{X}P,\notag\\
&&\bar{R}(X,P)P=R(X,P)P+\nabla_{X}P,\notag\\
&&\bar{S}(Y,Z)=S(Y,Z)-(n-2)g(\nabla_{Y}P,Z)+(n-2)\pi(Y)\pi(Z)-(n-2)g(Y,Z).\notag\\
&&\bar{S}(Y,P)=S(Y,P)=S(P,Y)=\bar{S}(P,Y).\notag
\end{eqnarray}

And if $P$ is a unit parallel vector field with respect to $\nabla$,
\begin{eqnarray}\label{25}
&&g(AX,Y)=g(AY,X),\quad AX=-\pi(X)P+\frac{1}{2}X,\notag\\
&&\bar{R}(X,Y,Z,W)=-\bar{R}(X,Y,W,Z),\notag\\
&&\bar{R}(X,Y)Z+\bar{R}(Y,Z)X+\bar{R}(Z,X)Y=0,\notag\\
&&R(P,X,Y,Z)=R(X,Y,P,Z)=0,\notag\\
&&\bar{R}(X,Y)P=R(X,Y)P=0,\notag\\
&&\bar{R}(P,X,Y,Z)=\bar{R}(X,Y,P,Z)=0,\notag\\
&&\bar{S}(Y,Z)=S(Y,Z)+(n-2)\pi(Y)\pi(Z)-(n-2)\pi(P)g(Y,Z)=\bar{S}(Z,Y),\notag\\
&&\bar{S}(Y,P)=S(Y,P)=0.\notag
\end{eqnarray}

\section{ Semi-quasi-Einstein Manifolds} \setcounter{section}{3}
 \setcounter{equation}{0}
In this section, we define a new structure using the
semi-symmetric metric connection.

\begin{defn}\label{defn31}
 A Riemannian manifold $(M^{n},\bar{\nabla})$, $n=dimM\geq 3$, is said to be a
semi-quasi-Einstein manifold S$(QE^{n})$ if the Ricci curvature tensor components are non-zero and of the forms
$$\hat{S}(Y,Z)=sym\bar{S}(Y,Z)=ag(Y,Z)+b\eta(Y)\eta(Z),$$
where $a$ and $b$ are scalars and $\eta$ is a non-zero 1-form. If $\hat{S}(Y,Z)$ is identically zero, we say the manifold $(M^{n},\bar{\nabla})$ is semi-Ricci flat. 
\end{defn}

\begin{thm}\label{thm31}
If the generator $P$  of manifold $(M^{n},\bar{\nabla})$ is a
Killing vector field with respect to $\nabla$, then an Einstein
manifold $(M^{n},\nabla)$ is a $S(QE^{n})$manifold. 
\end{thm}
\begin{proof} In view of the relation (\ref{24}),
\begin{eqnarray*}
\hat{S}(Y,Z)&=&\frac{1}{2}\{\bar{S}(Y,Z)+\bar{S}(Z,Y)\}\\
&=&S(Y,Z)-(n-2)[(\nabla_{Y}\pi)(Z)+(\nabla_{Z}\pi)(Y)]+(n-2)\pi(Y)\pi(Z)\\
&&-\{(n-2)\pi(P)+\sum_{i}g(\nabla_{i}P,e_{i})\}g(Y,Z).
\end{eqnarray*}
Notice that $(\nabla_{Y}\pi)(Z)+(\nabla_{Z}\pi)(Y)=g(\nabla_{Y}P,Z)+g(\nabla_{Z}P,Y),$
if $P$ is a Killing vector field with respect to $\nabla$, then one has
\begin{eqnarray}
\hat{S}(Y,Z)=S(Y,Z)+(n-2)\pi(Y)\pi(Z)-(n-2)\pi(P)g(Y,Z).\label{22}
\end{eqnarray}
Therefore, if $(M^{n},\nabla)$ is Einstein, then $(M^{n},\bar{\nabla})$ is a $S(QE^{n})$manifold.
\end{proof}

\begin{rem}
The above result is also true if we replace the Einstein manifold with a Ricci-flat manifold.
\end{rem}
\begin{rem}
If the generator $P$ is a unit parallel vector, $\hat{S}(X,Y)=\bar{S}(X,Y)$.
\end{rem}
\begin{thm}
For a unit vector field $P$, $(M^{n},\bar{\nabla})$ is a $S(QE^{n})$ manifold, if for some scalar $\rho$
\begin{eqnarray}\label{33}
\hat{S}(Y,Z)\hat{S}(X,W)-\hat{S}(X,Z)\hat{S}(Y,W)=\rho(g(Y,Z)g(X,W)-g(X,Z)g(Y,W)).
\end{eqnarray}

\end{thm}
\begin{proof}
Taking $X=W=P$ in (\ref{33}) we have
\begin{eqnarray*}
\hat{S}(Y,Z)\hat{S}(P,P)-\hat{S}(P,Z)\hat{S}(Y,P)=\rho\{g(Y,Z)-g(P,Y)g(P,Z)\}
\end{eqnarray*}
We denote by
$\alpha=\hat{S}(P,P),\,\hat{S}(P,Z)=g(\hat{Q}Z,P)=\pi(\hat{Q}Z)=\eta(Z)$, then
\begin{eqnarray}\label{34}
\hat{S}(Y,Z)=\frac{\rho}{\alpha}g(Y,Z)-\frac{\rho}{\alpha}\pi(Y)\pi(Z)+\frac{1}{\alpha}\eta(Y)\eta(Z).
\end{eqnarray} 
On the other hand, putting $X=P$ in (\ref{33}) we have
\begin{eqnarray}\label{35}
\hat{S}(Y,Z)\eta(W)-\hat{S}(Y,W)\eta(Z)=\rho\{g(Y,Z)\pi(W)-g(Y,W)\pi(Z)\}
\end{eqnarray}
Substituting (\ref{34}) into (\ref{35}) we obtain
\begin{eqnarray}\label{36}
&&\frac{\rho}{\alpha}\{g(Y,Z)\eta(W)-g(Y,W)\eta(Z)\}-\frac{\rho}{\alpha}\{\eta(W)\eta(Z)\pi(Y)-\pi(W)\pi(Y)\eta(Z)\}\nonumber\\
&&=\rho\{g(Y,Z)\pi(W)-g(Y,W)\pi(Z)\}
\end{eqnarray}

Now in (\ref{36}) , put $Y=Z=e_{i}$ and take sum about $i$, we get
\begin{equation}\label{37}
\eta(W)=\alpha\pi(W)
\end{equation}
we further obtain by substituting (\ref{37}) into (\ref{34})
\begin{eqnarray*}
\hat{S}(Y,Z)=\frac{\rho}{\alpha}g(Y,Z)+(\alpha-\frac{\rho}{\alpha})\pi(Y)\pi(Z).
\end{eqnarray*}
The proof is finished.
\end{proof}


 \section{$(M^n,\bar{\nabla})$-manifold with unit generator}
\setcounter{section}{4}
 \setcounter{equation}{0}
In this section we consider a $(M^n,\bar{\nabla})$-manifold with unit generator $\| P\|=1$.
\begin{defn}\label{defn32}
$(M^{n},\bar{\nabla})$ is said to be $\bar{\nabla}$-symmetric if the curvature tensor satisfies $\bar{\nabla}\bar{R}=0$; $\bar{\nabla}$-Ricci symmetric if $\bar{\nabla}\hat{S}=0$; semi-symmetric if $\bar{R}\cdot \bar{R}=0$; Ricci semi-symmetric if $ \bar{R}\cdot\hat{S}=0$.
\end{defn}
If $H$ is a $(0,k)$-type tensor field, we define the operation $R.H$
by
\begin{eqnarray*}\label{43}
(R(X,Y)\cdot H)(W_{1},\cdots,W_{k})&=&-H(R(X,Y)W_{1},\cdots,W_{k})\\
&&-\cdots-H(W_{1},\cdots,R(X,Y)W_{k}).
\end{eqnarray*}
If $\theta$ is a symmetric $(0,2)$-type tensor, the following
formula can define a $(0,k+2)$-type tensor
\begin{eqnarray}\label{44}
Q(\theta,H)(W_{1},\cdots,W_{k};X,Y)&=&-H((X\wedge_{\theta}Y)W_{1},\cdots,W_{k})\nonumber\\
&&-\cdots-H(W_{1},\cdots,(X\wedge_{\theta}Y)W_{k}),
\end{eqnarray}
where $X\wedge_{\theta}Y$ is given by
$(X\wedge_{\theta}Y)Z=\theta(Y,Z)X-\theta(X,Z)Y.$

When $P$ is a unit Killing field, based on the relation (\ref{22})
we get
\begin{eqnarray*}
\hat{S}(Y,Z)=S(Y,Z)+(n-2)\pi(Y)\pi(Z)-(n-2)g(Y,Z).
\end{eqnarray*}
By direct calculations we obtain
\begin{eqnarray}\label{41}
(\bar{\nabla}_{X}\hat{S})(Y,Z)&=&(\nabla_{X}S)(Y,Z)+(n-2)\{(\nabla_{X}\pi)(Y)\pi(Z)+(\nabla_{X}\pi)(Z)\pi(Y)\}\nonumber\\
&&-\pi(Y)[S(X,Z)+(n-2)\pi(X)\pi(Z)-(n-2)g(X,Z)]\nonumber\\
&&-\pi(Z)[S(Y,X)+(n-2)\pi(X)\pi(Y)-(n-2)g(X,Y)]\nonumber\\
&&+\pi(QY)g(X,Z)+\pi(QZ)g(X,Y)\\
&=&(\nabla_{X}S)(Y,Z)\nonumber\\
&&+g(X,Z)\bar{S}(Y,P)-g(Y,P)\bar{S}(X,Z)+g(X,Y)\bar{S}(Z,P)-g(Z,P)\bar{S}(X,Y)\nonumber
\end{eqnarray}
which derivates the following:
\begin{thm}\label{thm41}
If the generator $P$ is unit  Killing, $(M^{n},\nabla)$ is Ricci symmetric if and only if 
\begin{eqnarray*}
(\bar{\nabla}_{X}\hat{S})(Y,Z)=g(X,Z)\bar{S}(Y,P)-g(Y,P)\bar{S}(X,Z)+g(X,Y)\bar{S}(Z,P)-g(Z,P)\bar{S}(X,Y).
\end{eqnarray*}
\end{thm}
\begin{corol}
Let $P$ be a unit parallel vector field, then $(M^{n},\nabla)$ is Ricci symmetric if and only if
\begin{eqnarray*}
(\bar{\nabla}_{X}\hat{S})(Y,Z)=-\pi(Y)\bar{S}(X,Z)-\pi(Z)\bar{S}(X,Y).
\end{eqnarray*}
\end{corol}
Alternately, for a specific form of the Ricci curvature $S(X,Y)$ we have the following:
\begin{thm}\label{thm42}
Let $S(Y,Z)=(n-2)g(Y,Z)-(n-2)\pi(Y)\pi(Z),$ $\|P\|=1$. If the manifold $(M^{n},\nabla)$ is Ricci symmetric then $(M^{n},\bar{\nabla})$ is $\bar{\nabla}$-Ricci symmetric.
\end{thm}
\begin{proof}
Let $S(Y,Z)=(n-2)g(Y,Z)-(n-2)\pi(Y)\pi(Z)$ and $(\nabla_xS)(Y,Z)=0$. These together imply $(\nabla_X\pi)(Y)\pi(Z)+(\nabla_X\pi)(Z)\pi(Y)=0$ which means $\nabla P=0$ for a unit vector field $P$. Hence, $\bar{S}(Y,P)=0$. Therefore, using (\ref{41}) we conclude
\begin{eqnarray*}
(\bar{\nabla}_{X}\hat{S})(Y,Z)=(\nabla_{X}S)(Y,Z)&-&\pi(Y)[S(X,Z)+(n-2)\pi(X)\pi(Z)-(n-2)g(X,Z)]\\
&-&\pi(Z)[S(Y,X)+(n-2)\pi(X)\pi(Y)-(n-2)g(X,Y)]
\end{eqnarray*}
Hence the result.
\end{proof}

\begin{thm}\label{thm}
In general, if $(M^{n},\bar{\nabla})$ is a $\bar{\nabla}$-Ricci symmetric $S(QE^{n})$ manifold, $\hat{S}(X,Y)=ag(X,Y)+b\eta(X)\eta(Y)$ with a unit vector $\xi$ associated to the one-form $\eta$, then $a+b=constant$.
\end{thm}
\begin{proof}
We have
\begin{eqnarray}\label{42}
(\bar{\nabla}_{X}\hat{S})(Y,Z)&=&X(a)g(Y,Z)+X(b)\eta(Y)\eta(Z)\nonumber\\
&&+b[(\bar\nabla_{X}\eta)(Y)\eta(Z)+(\bar\nabla_{X}\eta)(Z)\eta(Y)].
\end{eqnarray}
Putting $Y=Z=\xi$ in (\ref{42}), we finish the proof.
\end{proof}\begin{thm}\label{thm44}
Let the generator $P$ in $(M^{n},\bar{\nabla})$ be a unit parallel vector field and $\bar{R}\cdot\bar{S}=-Q(g,\bar{S})$, if $(M^{n},\nabla)$ is Ricci semi-symmetric then $M$ is a quasi-Einstein manifold of the form $$S(X,Y)=(n-2)g(X,Y)-(n-2)\pi(X)\pi(Y).$$
Conversely, if the Ricci curvature satisfies $S(X,Y)=(n-2)g(X,Y)-(n-2)\pi(X)\pi(Y)$, then $M$ is Ricci semi-symmetric.
\end{thm}

\begin{proof}
we have
\begin{eqnarray*}
-(\bar{R}(X,Y)\cdot\bar{S})(Z,W)&=&\bar{S}(\bar{R}(X,Y)Z,W)+\bar{S}(Z,\bar{R}(X,Y)W)\\
&=&S(R(X,Y)Z,W)+S(Z,R(X,Y)W)\\
&&+g(AX,Z)S(Y,W)+g(AX,W)S(Y,Z)\\
&&-g(AY,Z)S(X,W)-g(AY,W)S(X,Z)\\
&&+g(X,Z)S(AY,W)+g(X,W)S(AY,Z)\\
&&-g(Y,Z)S(AX,W)-g(Y,W)S(AX,Z)\\
&&+(n-2)[\pi(W)\bar{R}(X,Y,Z,P)+\pi(Z)\bar{R}(X,Y,W,P)]\\
&=&-(R(X,Y)\cdot S)(Z,W)+Q(g,S)(Z,W;X,Y)\\
&&+\pi(Y)\pi(W)S(X,Z)-\pi(X)\pi(W)S(Y,Z)\\
&&+\pi(Y)\pi(Z)S(X,W)-\pi(X)\pi(Z)S(Y,W)\\
&=&-(R(X,Y)\cdot S)(Z,W)+Q(g,\bar{S})(Z,W;X,Y)\\
&&+[S(X,Z)-(n-2)g(X,Z)]\pi(Y)\pi(W)\\
&&-[S(Y,Z)-(n-2)g(Y,Z)]\pi(X)\pi(W)\\
&&+[S(X,W)-(n-2)g(X,W)]\pi(Y)\pi(Z)\\
&&-[S(Y,W)-(n-2)g(Y,W)]\pi(X)\pi(Z),
\end{eqnarray*}
where the last equality follows from 
\begin{eqnarray*}
Q(g,\bar{S})(Z,W;X,Y)&=&-\bar{S}((X\wedge Y)Z,W)-\bar{S}(Z,(X\wedge Y)W)\\
&=&Q(g,S)(Z,W;X,Y)+(n-2)[g(X,Z)\pi(Y)\pi(W)\\
&&-g(Y,Z)\pi(X)\pi(W)+g(X,W)\pi(Y)\pi(Z)\\
&&-g(Y,W)\pi(X)\pi(Z)].
\end{eqnarray*}
Since $\bar{R}\cdot\bar{S}=-Q(g,\bar{S})$, we have
\begin{eqnarray}
(R(X,Y)\cdot S)(Z,W)&=&[S(X,Z)-(n-2)g(X,Z)]\pi(Y)\pi(W)\notag\\
&-&[S(Y,Z)-(n-2)g(Y,Z)]\pi(X)\pi(W)\notag\\
&+&[S(X,W)-(n-2)g(X,W)]\pi(Y)\pi(Z)\notag\\
&-&[S(Y,W)-(n-2)g(Y,W)]\pi(X)\pi(Z).
\end{eqnarray}
Since $R\cdot S=0$, using $Y=W=P$ in the above equation we get the result. The converse part is quite obvious at this point.

\end{proof}


\section{ Some examples}

\begin{examp}
The first example we consider is the popular four dimensional Schwarzschild spacetime $M$ with the Ricci-flat metric \[g = \left( \frac{2 \, m}{r} - 1 \right) \mathrm{d} t\otimes \mathrm{d} t + \left( -\frac{1}{\frac{2 \, m}{r} - 1} \right) \mathrm{d} r\otimes \mathrm{d} r + r^{2} \mathrm{d} {\theta}\otimes \mathrm{d} {\theta} + r^{2} \sin\left({\theta}\right)^{2} \mathrm{d} {\phi}\otimes \mathrm{d} {\phi},\]
and a Killing vector $\frac{\partial}{\partial t}$ in this spacetime.
\end{examp}

The corresponding nontrivial Levi-Civita connection components are:
\begin{eqnarray*}
\begin{cases}
\Gamma_{ \phantom{\, t} \, t \, r }^{ \, t \phantom{\, t} \phantom{\, r} } = -\frac{m}{2 \, m r - r^{2}} , \Gamma_{ \phantom{\, r} \, t \, t }^{ \, r \phantom{\, t} \phantom{\, t} } = -\frac{2 \, m^{2} - m r}{r^{3}} , \Gamma_{ \phantom{\, r} \, r \, r }^{ \, r \phantom{\, r} \phantom{\, r} }  =  \frac{m}{2 \, m r - r^{2}} , \Gamma_{ \phantom{\, r} \, {\theta} \, {\theta} }^{ \, r \phantom{\, {\theta}} \phantom{\, {\theta}} }= 2 \, m - r , \Gamma_{ \phantom{\, {\phi}} \, {\theta} \, {\phi} }^{ \, {\phi} \phantom{\, {\theta}} \phantom{\, {\phi}} } = \frac{\cos\left({\theta}\right)}{\sin\left({\theta}\right)},\\ \Gamma_{ \phantom{\, r} \, {\phi} \, {\phi} }^{ \, r \phantom{\, {\phi}} \phantom{\, {\phi}} } = {\left(2 \, m - r\right)} \sin\left({\theta}\right)^{2} , \Gamma_{ \phantom{\, {\theta}} \, r \, {\theta} }^{ \, {\theta} \phantom{\, r} \phantom{\, {\theta}} } = \frac{1}{r} , \Gamma_{ \phantom{\, {\theta}} \, {\phi} \, {\phi} }^{ \, {\theta} \phantom{\, {\phi}} \phantom{\, {\phi}} } = -\cos\left({\theta}\right) \sin\left({\theta}\right) , \Gamma_{ \phantom{\, {\phi}} \, r \, {\phi} }^{ \, {\phi} \phantom{\, r} \phantom{\, {\phi}} } = \frac{1}{r} . 
\end{cases}
\end{eqnarray*}


We define a semi-symmetric metric connection corresponding to the one-form $\pi=( \frac{2m - r}{r})dt$ associated to the Killing vector $\frac{\partial}{\partial t}$ by
\begin{equation*}
\bar{\Gamma}_{ij}^{k}=\Gamma_{ij}^{k}+\pi_{j}\delta_{i}^{k}-g_{ij}\pi^{k},
\pi^{k}=g^{ik}\pi_{i}.
\end{equation*}
We calculate the nontrivial components as:
\begin{eqnarray*} 
\begin{cases}\bar{\Gamma}_{ \phantom{\, t} \, t \, r }^{ \, t \phantom{\, t} \phantom{\, r} }  =  -\frac{m}{2 \, m r - r^{2}}, \bar{\Gamma}_{ \phantom{\, t} \, r \, t }^{ \, t \phantom{\, r} \phantom{\, t} }  =  -\frac{m}{2 \, m r - r^{2}}, \bar{\Gamma}_{ \phantom{\, t} \, r \, r }^{ \, t \phantom{\, r} \phantom{\, r} }  =  \frac{r}{2 \, m - r}, \bar{\Gamma}_{ \phantom{\, t} \, {\theta} \, {\theta} }^{ \, t \phantom{\, {\theta}} \phantom{\, {\theta}} }  =  -r^{2}, \bar{\Gamma}_{ \phantom{\, t} \, {\phi} \, {\phi} }^{ \, t \phantom{\, {\phi}} \phantom{\, {\phi}} } = -r^{2} \sin\left({\theta}\right)^{2},\\
 \bar{\Gamma}_{ \phantom{\, r} \, t \, t }^{ \, r \phantom{\, t} \phantom{\, t} }  =  -\frac{2 \, m^{2} - m r}{r^{3}} , \bar{\Gamma}_{ \phantom{\, r} \, r \, t }^{ \, r \phantom{\, r} \phantom{\, t} }  =  \frac{2 \, m - r}{r} ,\bar{\Gamma}_{ \phantom{\, r} \, r \, r }^{ \, r \phantom{\, r} \phantom{\, r} }  =  \frac{m}{2 \, m r - r^{2}} , \bar{\Gamma}_{ \phantom{\, r} \, {\theta} \, {\theta} }^{ \, r \phantom{\, {\theta}} \phantom{\, {\theta}} }  =  2 \, m - r , \bar{\Gamma}_{ \phantom{\, r} \, {\phi} \, {\phi} }^{ \, r \phantom{\, {\phi}} \phantom{\, {\phi}} }  =  {\left(2 \, m - r\right)} \sin\left({\theta}\right)^{2}, \\ 
 \bar{\Gamma}_{ \phantom{\, {\theta}} \, r \, {\theta} }^{ \, {\theta} \phantom{\, r} \phantom{\, {\theta}} }  =  \frac{1}{r} , \bar{\Gamma}_{ \phantom{\, {\theta}} \, {\theta} \, t }^{ \, {\theta} \phantom{\, {\theta}} \phantom{\, t} }  =  \frac{2 \, m - r}{r} , \bar{\Gamma}_{ \phantom{\, {\theta}} \, {\theta} \, r }^{ \, {\theta} \phantom{\, {\theta}} \phantom{\, r} }  =  \frac{1}{r} , \bar{\Gamma}_{ \phantom{\, {\theta}} \, {\phi} \, {\phi} }^{ \, {\theta} \phantom{\, {\phi}} \phantom{\, {\phi}} }  =  -\cos\left({\theta}\right) \sin\left({\theta}\right) , \bar{\Gamma}_{ \phantom{\, {\phi}} \, r \, {\phi} }^{ \, {\phi} \phantom{\, r} \phantom{\, {\phi}} } =  \frac{1}{r} , \bar{\Gamma}_{ \phantom{\, {\phi}} \, {\theta} \, {\phi} }^{ \, {\phi} \phantom{\, {\theta}} \phantom{\, {\phi}} } = \frac{\cos\left({\theta}\right)}{\sin\left({\theta}\right)} ,\\
  \bar{\Gamma}_{ \phantom{\, {\phi}} \, {\phi} \, t }^{ \, {\phi} \phantom{\, {\phi}} \phantom{\, t} }  =  \frac{2 \, m - r}{r} , \bar{\Gamma}_{ \phantom{\, {\phi}} \, {\phi} \, r }^{ \, {\phi} \phantom{\, {\phi}} \phantom{\, r} }  =  \frac{1}{r} , \bar{\Gamma}_{ \phantom{\, {\phi}} \, {\phi} \, {\theta} }^{ \, {\phi} \phantom{\, {\phi}} \phantom{\, {\theta}} }  =  \frac{\cos\left({\theta}\right)}{\sin\left({\theta}\right)}. 
\end{cases}
\end{eqnarray*}

The corresponding Ricci curvature components $\bar{S}_{ij}$ are:
\begin{center}
$\left(\begin{array}{rrrr}
0 & -\frac{m}{r^{2}} & 0 & 0 \\
\frac{m}{r^{2}} & 1 & 0 & 0 \\
0 & 0 & r^{2}-2 m r  & 0 \\
0 & 0 & 0 & -{\left(2 \, m r - r^{2}\right)} \sin\left({\theta}\right)^{2}
\end{array}\right)$
\end{center}

For $a=\frac{r-2m}{r}$ and $b=1$, we check that:
\begin{eqnarray*}
0=\hat{S}_{tt} &=& ag_{tt}+b\pi_t\pi_t\\
1=\hat{S}_{rr} &=& ag_{rr}+b\pi_r\pi_r\\
r^2-2mr=\hat{S}_{\theta\theta} &=& ag_{\theta\theta}+b\pi_\theta\pi_\theta\\
-(2mr-r^2)\sin^2\theta=\hat{S}_{\phi\phi} &=& ag_{\phi\phi}+b\pi_\phi\pi_\phi
\end{eqnarray*}
Hence, $(M,\bar{\nabla})$ is a $S(QE^n)$ manifold.

\begin{examp}
We can further consider more complicated 4-dimenisonal Kottler spacetime with metric \cite{kottler}
\begin{equation*}
g = \left( \frac{1}{3} \, \Lambda r^{2} + \frac{2 \, m}{r} - 1 \right) \mathrm{d} t\otimes \mathrm{d} t + \left( -\frac{3}{\Lambda r^{2} + \frac{6 \, m}{r} - 3} \right) \mathrm{d} r\otimes \mathrm{d} r + r^{2} \mathrm{d} {\theta}\otimes \mathrm{d} {\theta} + r^{2} \sin\left({\theta}\right)^{2} \mathrm{d} {\phi}\otimes \mathrm{d} {\phi}
\end{equation*}
as an example of an Einstein space carrying a $S(QE^n)$  structure with a Killing generator. This spacetime satisfies  Einstein's  field  equations  of  general  relativity  with  positive  cosmological  constant for a  vacuum space around a center of spherical symmetry. The manifold is also called the Schwarzschild-de Sitter spacetime and provides us with a two-parameter family of static spacetime with compact spacelike slices, which are locally (but not globally) conformally flat. The Levi-Civita components of this metric are:
\begin{eqnarray*} 
\begin{cases}\Gamma_{ \phantom{\, t} \, t \, r }^{ \, t \phantom{\, t} \phantom{\, r} } =  \frac{\Lambda r^{3} - 3 \, m}{\Lambda r^{4} + 6 \, m r - 3 \, r^{2}}, \Gamma_{ \phantom{\, r} \, t \, t }^{ \, r \phantom{\, t} \phantom{\, t} }  =  \frac{\Lambda^{2} r^{6} + 3 \, \Lambda m r^{3} - 3 \, \Lambda r^{4} - 18 \, m^{2} + 9 \, m r}{9 \, r^{3}},  \Gamma_{ \phantom{\, r} \, r \, r }^{ \, r \phantom{\, r} \phantom{\, r} } =  -\frac{\Lambda r^{3} - 3 \, m}{\Lambda r^{4} + 6 \, m r - 3 \, r^{2}}\\ \Gamma_{ \phantom{\, r} \, {\theta} \, {\theta} }^{ \, r \phantom{\, {\theta}} \phantom{\, {\theta}} }  =  \frac{1}{3} \, \Lambda r^{3} + 2 \, m - r, \Gamma_{ \phantom{\, r} \, {\phi} \, {\phi} }^{ \, r \phantom{\, {\phi}} \phantom{\, {\phi}} } = \frac{1}{3} \, {\left(\Lambda r^{3} + 6 \, m - 3 \, r\right)} \sin\left({\theta}\right)^{2}, \Gamma_{ \phantom{\, {\theta}} \, r \, {\theta} }^{ \, {\theta} \phantom{\, r} \phantom{\, {\theta}} } =\frac{1}{r}\\ \Gamma_{ \phantom{\, {\theta}} \, {\phi} \, {\phi} }^{ \, {\theta} \phantom{\, {\phi}} \phantom{\, {\phi}} }  =  -\cos\left({\theta}\right) \sin\left({\theta}\right), \Gamma_{ \phantom{\, {\phi}} \, r \, {\phi} }^{ \, {\phi} \phantom{\, r} \phantom{\, {\phi}} } =  \frac{1}{r}, \Gamma_{ \phantom{\, {\phi}} \, {\theta} \, {\phi} }^{ \, {\phi} \phantom{\, {\theta}} \phantom{\, {\phi}} } = \frac{\cos\left({\theta}\right)}{\sin\left({\theta}\right)} 
\end{cases}
\end{eqnarray*}
The Ricci tensor $S_{ij}$ is calculated as 
\begin{center}
$\left(\begin{array}{rrrr}
-\frac{\Lambda^{2} r^{3} + 6 \, \Lambda m - 3 \, \Lambda r}{3 \, r} & 0 & 0 & 0 \\
0 & \frac{3 \, \Lambda r}{\Lambda r^{3} + 6 \, m - 3 \, r} & 0 & 0 \\
0 & 0 & -\Lambda r^{2} & 0 \\
0 & 0 & 0 & -\Lambda r^{2} \sin\left({\theta}\right)^{2}
\end{array}\right)$
\end{center}
Clearly $S_{ij}=-\Lambda g_{ij}$ in this spacetime. Hence it is an Einstein space.

The semi-symmetric metric connection corresponding to the one-form $\pi = (\frac{\Lambda r^{3} + 6 \, m - 3 \, r}{3 \, r})dt$ associated to the Killing vector $\frac{\partial}{\partial t}$ is given by 
\begin{equation*}
\bar{\Gamma}_{ij}^{k}=\Gamma_{ij}^{k}+\pi_{j}\delta_{i}^{k}-g_{ij}\pi^{k},
\pi^{k}=g^{ik}\pi_{i}.
\end{equation*}
We calculate its nontrivial components as:
\begin{eqnarray*} 
\begin{cases}
\bar{\Gamma}_{ \phantom{\, t} \, t \, r }^{ \, t \phantom{\, t} \phantom{\, r} } = \frac{\Lambda r^{3} - 3 \, m}{\Lambda r^{4} + 6 \, m r - 3 \, r^{2}}, \bar{ \Gamma}_{ \phantom{\, t} \, r \, t }^{ \, t \phantom{\, r} \phantom{\, t} }= \frac{\Lambda r^{3} - 3 \, m}{\Lambda r^{4} + 6 \, m r - 3 \, r^{2}}, \bar{\Gamma}_{ \phantom{\, t} \, r \, r }^{ \, t \phantom{\, r} \phantom{\, r} } = \frac{3 \, r}{\Lambda r^{3} + 6 \, m - 3 \, r}, \bar{\Gamma}_{ \phantom{\, t} \, {\theta} \, {\theta} }^{ \, t \phantom{\, {\theta}} \phantom{\, {\theta}} } =-r^{2}, \bar{\Gamma}_{ \phantom{\, t} \, {\phi} \, {\phi} }^{ \, t \phantom{\, {\phi}} \phantom{\, {\phi}} } = -r^{2} \sin\left({\theta}\right)^{2} \\ \bar{\Gamma}_{ \phantom{\, r} \, t \, t }^{ \, r \phantom{\, t} \phantom{\, t} } = \frac{\Lambda^{2} r^{6} + 3 \, \Lambda m r^{3} - 3 \, \Lambda r^{4} - 18 \, m^{2} + 9 \, m r}{9 \, r^{3}}, \bar{\Gamma}_{ \phantom{\, r} \, r \, t }^{ \, r \phantom{\, r} \phantom{\, t} } = \frac{\Lambda r^{3} + 6 \, m - 3 \, r}{3 \, r}, \bar{\Gamma}_{ \phantom{\, r} \, r \, r }^{ \, r \phantom{\, r} \phantom{\, r} } = -\frac{\Lambda r^{3} - 3 \, m}{\Lambda r^{4} + 6 \, m r - 3 \, r^{2}},\, \bar{\Gamma}_{ \phantom{\, {\phi}} \, {\phi} \, {\theta} }^{ \, {\phi} \phantom{\, {\phi}} \phantom{\, {\theta}} }=\frac{\cos\left({\theta}\right)}{\sin\left({\theta}\right)}\\ \bar{\Gamma}_{ \phantom{\, r} \, {\theta} \, {\theta} }^{ \, r \phantom{\, {\theta}} \phantom{\, {\theta}} }= \frac{1}{3} \, \Lambda r^{3} + 2 \, m - r, \bar{\Gamma}_{ \phantom{\, r} \, {\phi} \, {\phi} }^{ \, r \phantom{\, {\phi}} \phantom{\, {\phi}} } = \frac{1}{3} \, {\left(\Lambda r^{3} + 6 \, m - 3 \, r\right)} \sin\left({\theta}\right)^{2}, \bar{\Gamma}_{ \phantom{\, {\theta}} \, r \, {\theta} }^{ \, {\theta} \phantom{\, r} \phantom{\, {\theta}} } = \frac{1}{r}, \bar{\Gamma}_{ \phantom{\, {\theta}} \, {\theta} \, t }^{ \, {\theta} \phantom{\, {\theta}} \phantom{\, t} } = \frac{\Lambda r^{3} + 6 \, m - 3 \, r}{3 \, r} \\ \bar{\Gamma}_{ \phantom{\, {\theta}} \, {\phi} \, {\phi} }^{ \, {\theta} \phantom{\, {\phi}} \phantom{\, {\phi}} } = -\cos\left({\theta}\right) \sin\left({\theta}\right), \bar{\Gamma}_{ \phantom{\, {\phi}} \, r \, {\phi} }^{ \, {\phi} \phantom{\, r} \phantom{\, {\phi}} } = \frac{1}{r}, \bar{\Gamma}_{ \phantom{\, {\phi}} \, {\theta} \, {\phi} }^{ \, {\phi} \phantom{\, {\theta}} \phantom{\, {\phi}} } = \frac{\cos\left({\theta}\right)}{\sin\left({\theta}\right)}, \bar{\Gamma}_{ \phantom{\, {\phi}} \, {\phi} \, t }^{ \, {\phi} \phantom{\, {\phi}} \phantom{\, t} }= \frac{\Lambda r^{3} + 6 \, m - 3 \, r}{3 \, r}, \bar{\Gamma}_{ \phantom{\, {\phi}} \, {\phi} \, r }^{ \, {\phi} \phantom{\, {\phi}} \phantom{\, r} } = \frac{1}{r}, \bar{\Gamma}_{ \phantom{\, {\theta}} \, {\theta} \, r }^{ \, {\theta} \phantom{\, {\theta}} \phantom{\, r} } = \frac{1}{r}  
\end{cases}
\end{eqnarray*}
The corresponding Ricci curvature components $\bar{S}_{ij}$ are:
\begin{center}
$\left(\begin{array}{rrrr}
-\frac{\Lambda^{2} r^{3} + 6 \, \Lambda m - 3 \, \Lambda r}{3 \, r} & \frac{\Lambda r^{3} - 3 \, m}{3 \, r^{2}} & 0 & 0 \\
-\frac{\Lambda r^{3} - 3m}{3r^{2}} & \frac{\Lambda r^{3} + 3(\Lambda - 1) r + 6m}{\Lambda r^{3} + 6m - 3r} & 0 & 0 \\
0 & 0 & -\frac{\Lambda r^{4} + 3(\Lambda - 1) r^{2} + 6m r}{3} & 0 \\
0 & 0 & 0 & -\frac{(\Lambda r^4+3(\Lambda - 1)r^{2} + 6mr)}{3}\sin(\theta)^2
\end{array}\right)$
\end{center}

For $a=-\Lambda-\frac{\Lambda r^2+6m-3r}{3r}$ and $b=1$, we check that the nontrivial $\bar{S}_{ij}$ components satisfy:
\begin{eqnarray*}
\frac{\Lambda^{2} r^{3} + 6 \Lambda m - 3  \Lambda r}{3 r}=\hat{S}_{tt} &=& ag_{tt}+b\pi_t\pi_t\\
\frac{\Lambda r^{3} + 3(\Lambda - 1) r + 6m}{\Lambda r^{3} + 6m - 3r}=\hat{S}_{rr} &=& ag_{rr}+b\pi_r\pi_r\\
\frac{\Lambda r^{4} + 3(\Lambda - 1) r^{2} + 6m r}{3}=\hat{S}_{\theta\theta} &=& ag_{\theta\theta}+b\pi_\theta\pi_\theta\\
\frac{\Lambda r^{4} + 3(\Lambda - 1) r^{2} + 6m r}{3}\sin^2\theta=\hat{S}_{\phi\phi} &=& ag_{\phi\phi}+b\pi_\phi\pi_\phi
\end{eqnarray*}
Hence, $(M,\bar{\nabla})$ is a $S(QE^n)$ manifold.

\end{examp}
\begin{examp}
We consider a three-dimensional manifold $M^{3}$ endowed with a
metric
\begin{equation*}
g_{ij}dx^{i}dx^{j}=e^{x^{1}}[(dx^{1})^{2}-(dx^{2})^{2}]+(dx^{3})^{2}
\end{equation*}
and a non-zero 1-form $\pi=e^{\frac{1}{2}x^{1}-\frac{1}{2}x^{2}}dx^{1}-e^{\frac{1}{2}x^{1}-\frac{1}{2}x^{2}}dx^{2}+dx^{3}$.
\end{examp}
The non-zero components of the corresponding Levi-Civita connnection are $\Gamma_{11}^{1}=\Gamma_{12}^{2}=\Gamma_{22}^{1}=\frac{1}{2}$ and the manifold is Ricci flat. Furthermore, $\pi$ is a parallel one-form with respect to this connection. 

The semi-symmetric metric connection corresponding to the 1-form $\pi$ is given by
\begin{equation*}
\bar{\Gamma}_{ij}^{k}=\Gamma_{ij}^{k}+\pi_{j}\delta_{i}^{k}-g_{ij}\pi^{k},
\end{equation*}
and the non-zero coefficients are
\begin{eqnarray*}
\begin{cases}
\bar{\Gamma}_{11}^{1}=\frac{1}{2},\bar{\Gamma}_{11}^{2}=-e^{\frac{1}{2}x^{1}-\frac{1}{2}x^{2}},\bar{\Gamma}_{11}^{3}=-e^{x^{1}},\bar{\Gamma}_{12}^{1}=-e^{\frac{1}{2}x^{1}-\frac{1}{2}x^{2}},\bar{\Gamma}_{12}^{2}=\frac{1}{2},\bar{\Gamma}_{13}^{1}=1,\\
\bar{\Gamma}_{23}^{2}=1,\bar{\Gamma}_{21}^{2}=\frac{1}{2}+e^{\frac{1}{2}x^{1}-\frac{1}{2}x^{2}},\bar{\Gamma}_{22}^{1}=\frac{1}{2}+e^{\frac{1}{2}x^{1}-\frac{1}{2}x^{2}},\bar{\Gamma}_{22}^{3}=e^{x^{1}}, \bar{\Gamma}_{31}^{3}=e^{\frac{1}{2}x^{1}-\frac{1}{2}x^{2}},\\
\bar{\Gamma}_{32}^{3}=-e^{\frac{1}{2}x^{1}-\frac{1}{2}x^{2}},
\bar{\Gamma}_{33}^{1}=-e^{-\frac{1}{2}x^{1}-\frac{1}{2}x^{2}}=\bar{\Gamma}_{33}^{2}.
\end{cases}
\end{eqnarray*}

The Ricci curvature components $\bar{S_{ij}}$ satisfies
\begin{eqnarray*}
\begin{cases}
\hat{S}_{11}=-e^{x^{1}}+e^{x^{1}-x^{2}}=-g_{11}+\pi_{1}\pi_{1},\\
\hat{S}_{12}=-e^{x^{1}-x^{2}}=-g_{12}+\pi_{1}\pi_{2},\\
\hat{S}_{13}=e^{\frac{1}{2}x^{1}-\frac{1}{2}x^{2}}=-g_{13}+\pi_{1}\pi_{3},\\
\hat{S}_{22}=e^{x^{1}}+e^{x^{1}-x^{2}}=-g_{22}+\pi_{2}\pi_{2},\\
\hat{S}_{23}=-e^{\frac{1}{2}x^{1}-\frac{1}{2}x^{2}}=-g_{23}+\pi_{2}\pi_{3},\\
\hat{S}_{33}=0=-g_{33}+\pi_{3}\pi_{3},
\end{cases}
\end{eqnarray*}
now we take $a=-1$ and $b=1$, then
$$\bar{S}_{ij}=ag_{ij}+b\pi_{i}\pi_{j}, i,j=1,2,3$$
therefore $(M^{3},\bar{\nabla})$ is a $S(QE^{n})$ manifold.

\section{ Acknowledgments}

The third author would like to thank Professor H. Z. Li for his
encouragement and help!


\begin{thebibliography}{s20}


\bibitem{Be}C. L. Bejan, \emph{ Characterization of quasi-Einstein manifolds}. An.
\c{S}iint. Univ. Al. I. Cuza Ia\c{s}i. Mat. (N.S.). {\bf 53}(2007),
67-72.

\bibitem{CMC1} M.C. Chaki, R.K. Maity, \emph{ On quasi Einstein manifolds}.
Publ. Math. Debrecen. {\bf 57(3-4)}(2000), 297-306.

\bibitem{CMC2} M.C. Chaki, \emph{On generalized quasi Einstein
mamnifolds}. Publ. Math. Debrecen. {\bf 58(4)}(2001), 683-691.

\bibitem{CM} C.\"{O}zg\"{u}r, \emph{ N(k)-quasi Einstein manifolds satisfying certain
conditions}. Chaos Solitons Fractals. {\bf 38(5)}(2008), 1373-1377.

\bibitem{Mura} C.\"{O}zg\"{u}r, \emph{Riemannian  manifolds  with  a  semi-symmetric  metric  connection  satisfyingsome  semisymmetry  conditions},Proceedings of the Estonian Academy of Sciences, {\bf 57(4)}(2008), 210–216.

\bibitem{CS}C.\"{O}zg\"{u}r and S. Sular, \emph{ On N(k)-quasi Einstein manifolds satisfying
certain conditions}. Balkan J. Geom. Appl. {\bf 13(2)}(2008), 74-79.

\bibitem{DD1}U. C. De and B. K. De, \emph{ On quasi Einstein manifolds}. Commun. Korean
Math. Soc. {\bf 23(3)}(2008), 413-420.

\bibitem{DD2}U. C. De and G. C. Ghosh, \emph{ On conformally flat special quasi
Einstein manifolds}.  Publ. Math. Debrecen, {\bf 66(1-2)}(2005),
129-136.

\bibitem{DG1}U. C. De and G. C. Ghosh,\emph{ On quasi Einstein manifolds}. Period. Math. Hungar. {\bf 48(1-2)}(2004), 223-231.

\bibitem{DG2}U. C. De and G. C. Ghosh, \emph{ On Quasi Einstein and Special quasi Einstein manifolds} (Kuwait University Department of
Mathematics Computer Science, Kuwait, 2005), pp. 178-191.

\bibitem{FHZ}F.Y.Fu, Y.L.Han and P.B.Zhao,\emph{Geometric and physical characteristics of
mixed super quasi-Einstein manifolds}. INT. J. GEOM. METHODS M.,{\bf
16}(2019),1950104-1~15.

\bibitem{semi-rel} Gh. Fasihi-Ramandi, Semi-symmetric connection formalism for unification of gravity and electromagnetism, J.Geom.Phys. 144 (2019) 245-250.
\bibitem {TK} M.M.Tripathi and J.S.Kim, \emph{ On N(k)-quasi Einstein manifold}.
Commun. Korean Math. Soc.,{\bf 22(3)}(2007), 411-417.





\bibitem{Ki} V. A. Kiosak, \emph{ On the conformal mappings of quasi-Einstein
spaces}. J. Math. Sci. {\bf 184}(2012), 12-18.

\bibitem{Li1} Z.L.Li, \emph{ On quasi-Einstein manifolds, Journal of Hangzhou University}. {\bf
16(2)}(1989), 115-122.

\bibitem{Neil} B. O'Neill, \emph{Semi-Riemannian geometry. With applications to
relativity}. Pure and Applied Mathematics, 103. Academic Press, Inc.

\bibitem{schouten} J.A. Schouten, \emph{Ricci-Calculus}, second ed., Springer, Berlin, 1954.

\bibitem{Ya} K.Yano, \emph{ On semi-symmetric metric connection}, Rev. Roum.
Math. Pureset Appl. {\bf 15}(1970), 1579-1586.

\bibitem{kottler} P. G. LeFloch and L. Rozoy,\emph{Uniqueness of Kottler spacetime and the Besse conjecture}, Comptes Rendus Mathematique, Volume 348, Issues 19–20, October 2010, Pages 1129-1132.




\end{thebibliography}
\end{document}